\documentclass{amsart} 
\usepackage{amssymb}
\usepackage{mathabx}

\usepackage[usenames,pdftex]{color}

 \usepackage[colorlinks,citecolor=blue, linkcolor= blue,urlcolor=Violet]{hyperref}
	\usepackage{xcolor}
	\usepackage{soul}
	\definecolor{lightblue}{rgb}{.60,.60,1}
	\definecolor{lightred}{rgb}{1, .60,.60}

\usepackage{tikz}
\usetikzlibrary{arrows,positioning} 
\tikzset{
    >=stealth',
    imp/.style={
           ->,
           shorten <=2pt,
           shorten >=2pt,}
}

\newcommand{\seq}[1]{{\left\langle{#1}\right\rangle}}

\newcommand{\rest}[1]{\! \upharpoonright_{#1}}

\newcommand{\w}{\omega}
\newcommand \s{\sigma}
\renewcommand{\le}{\leqslant}
\renewcommand{\ge}{\geqslant}

\newcommand \Tur{\textup{\scriptsize T}}

\newcommand{\leb}{\lambda}

\newcommand{\UU}{{\mathcal U}}
\newcommand{\PP}{\mathcal P}
\newcommand{\CC}{\mathcal C}

\newcommand{\LRH}{\textup{\textsf{LRH}}}

\theoremstyle{plain}
\newtheorem{theorem}{Theorem}

\newtheorem{corollary}[theorem]{Corollary} 
\newtheorem{question}[theorem]{Question}

\theoremstyle{definition}

\theoremstyle{remark}

\title[Computing $K$-trivial sets by incomplete random sets]{Computing {\boldmath $K$}-trivial sets by\\ incomplete random sets}

\author[Bienvenu]{Laurent Bienvenu}

\author[Day]{Adam R.~Day}
\address{Department of Mathematics\\
University of California, Berkeley\\
Berkeley, CA 94720-3840, USA}
\email{adam.day@math.berkeley.edu}

\author[Greenberg]{Noam Greenberg}
\address{School of Mathematics, Statistics and Operations Research, Victoria University of Wellington, Wellington, New Zealand}
\email{greenberg@msor.vuw.az.nz}

\author[Ku\v{c}era]{Anton{\'\i}n Ku\v{c}era}
\address{Charles University in Prague,  Faculty of Mathematics and Physics, Prague, Czech Republic}
\email{kucera@mbox.ms.mff.cuni.cz}

\author[Miller]{Joseph S.~Miller}
\address{Department of Mathematics\\
University of Wisconsin\\
Madison, WI 53706-1388, USA}
\email{jmiller@math.wisc.edu}

\author[Nies]{Andr\'{e} Nies}

\address{Department of Computer Science, Private Bag 92019, Auckland, New Zealand}
\email{andre@cs.auckland.ac.nz}
\author[Turetsky]{Dan Turetsky}
\address{Kurt G\"odel Research Center for Mathematical Logic\\
University of Vienna\\
Vienna, Austria}
\email{turetsd4@univie.ac.at}

\begin{document} 	

\begin{abstract}
Every $K$-trivial set is computable from an incomplete Martin-L\"{o}f random set, i.e., a Martin-L\"{o}f random set that does not compute $\emptyset'$. 
\end{abstract}

\thanks{
Day was supported by a Miller Research Fellowship in the Department of Mathematics at the University of California, Berkeley.
Greenberg was supported by the Marsden Fund of New Zealand, and by a Rutherford Discovery Fellowship. Miller was supported by the National Science Foundation under grant DMS-1001847. Nies was supported by the Marsden Fund. Turetsky was supported by the Marsden Fund, via a postdoctoral scholarship. A major part of the work by Bienvenu, Greenberg, Ku\v{c}era, Nies and Turetsky was performed during a research-in-pairs stay at the Mathematisches Forschungsinstitut Oberwolfach. 
}

\maketitle

A major objective in algorithmic randomness is to understand how random sets   and computably enumerable (c.e.)  sets
 interact  within the Turing degrees. 
 At some level of randomness all interesting  interactions cease. Computably enumerable sets are each definable and, in a sense, they are as effective as any noncomputable set can be. 
 As a consequence, the lower and upper cones of noncomputable c.e.\ sets are definable null sets, and thus if a set is ``sufficiently'' random, it cannot compute, nor be computed by, a noncomputable c.e.\ set. 
However,  the most studied notion of algorithmic randomness, Martin-L\"{o}f randomness, is not strong enough to support 
this argument, and in fact, significant interactions between Martin-L\"{o}f random sets and c.e.\ sets occur. The study of these interactions has lead to a number of surprising results that show a remarkably robust relationship between Martin-L\"{o}f random sets and the class of $K$-trivial sets. Interestingly, the significant interaction occurs ``at the boundaries'': the Martin-L\"{o}f random sets in question are close to being non-random (in that they fail fairly simple statistical tests), and $K$-trivial c.e.\ sets are close to being computable.

The following theorem resolves one of the main open questions in algorithmic randomness, and further strengthens the relationship between the Martin-L\"of random sets and the $K$-trivial sets.
\begin{theorem} \label{thm_main}
There is an incomplete Martin-L\"{o}f random set  that computes every $K$-trivial set.
\end{theorem}
This theorem is essentially a corollary of two recent results, both proved in 2012: the first by Bienvenu, Greenberg, Ku\v{c}era, Nies and Turetsky~\cite{Bienvenu.Greenberg.ea:preprint}; and the second by Day and Miller~\cite{Day.Miller:nd}. In the remainder of this announcement, we will explain the background to the problem behind this theorem,  and indicate the main ideas used in the proof.

In this announcement, by ``random'' we will henceforth mean Martin-L\"{o}f random. We will give the full definition shortly, but essentially, an element $X$ of Cantor space is Martin-L\"{o}f random if it is not an element of a particular kind of effectively presented,  null $G_\delta$ class.\footnote{Elements of Cantor space are called \emph{reals}, \emph{sequences} or simply \emph{sets} by computability theorists, the latter since they are identified with subsets of $\w$. Subsets of Cantor space are referred to as either \emph{sets} or \emph{classes}.}
Any random set computes a diagonally noncomputable function. This implies, by Arslanov's completeness criterion~\cite{Arslanov:1981}, that the only c.e.\ sets that can compute random sets are the complete ones. 
Two questions of interest are  thus: 

\begin{itemize} \item[(1)] classifying those random sets that can compute incomputable c.e.\ sets, and

\item[(2)] classifying  those c.e.\ sets that can be computed by random sets. \end{itemize}
In 1986, Ku\v{c}era~\cite{Kucera:86} showed that every $\Delta^0_2$ random set is Turing above  a noncomputable c.e.\ set, demonstrating that these questions may have interesting  answers. 


The first question was settled by Hirschfeldt and Miller (see \cite[Th.~5.3.16]{Nies:Book}). To explain their answer we need to provide some definitions. We denote Cantor space by $2^\w$. An effectively presented $G_\delta$ set is the  intersection $\bigcap_n \UU_n$ of a nested, computable sequence of effectively open ($\Sigma^0_1$) subsets of $2^\w$  (i.e., it is a $\Pi^0_2$ class). A $G_\delta$ set $E$ is \emph{null} if $\leb (E) =0$, where $\leb$ denotes the fair-coin measure on Cantor space, which is mapped to Lebesgue measure under the standard near-isomorphism between Cantor space and the unit interval $[0,1]$. Those elements of Cantor space that are not contained in  any null, effectively presented $G_\delta$ sets are, in the parlance of the field, known as \emph{weakly $2$-random}. 
A set $X \in 2^\w$ is \emph{Martin-L\"of random} if it is not a member of any effectively presented $G_\delta$ set $\bigcap_n \UU_n$ with $\leb(\UU_n)\le 2^{-n}$. This additional condition specifies that not only is  $\bigcap_n \UU_n$  null, but this fact is witnessed in a strong manner. 
Such a null set is known as a \emph{test} for Martin-L\"of randomness. Martin-L\"{o}f randomness can be defined by other means, for example using effective betting strategies or compressibility of initial segments.\footnote{For more on algorithmic randomness, the reader can consult the books \cite{Nies:Book,DowneyHirschfeldtBook}, which are the most up-to-date surveys of the field.}
  Hirschfeldt and Miller showed that a Martin-L\"of random set $X$ computes a noncomputable c.e.\ set if and only if it is not weakly $2$-random. This gives a pleasing characterization of those ``special'' random sets that compute noncomputable c.e.\ sets using the tools of effective measure theory (or very low level effective descriptive set theory). 

For the second question, we are faced with a simple example: \emph{all} c.e.\ sets are computable from some random set; indeed a single random set, Chaitin's ``halting probability'' $\Omega$, computes all c.e.\ sets. 
In fact, the Ku\v{c}era--G\'{a}cs theorem \cite{Kucera:85,Gacs:86} states that every set is computable from a random set. 
However, this is only a partial answer to the question. A theorem of Stephan's \cite{Stephan:06} establishes a dichotomy between two kinds of random sets. 
On the one hand, those randoms that compute $\emptyset'$, the \emph{complete} ones,  pass all the relevant statistical tests (the effective null classes) not because they are in some way typical, but because their strong information content allows them mimic typical sets.
On the other hand, \emph{incomplete} random sets, those that do not compute $\emptyset'$, are ``more random'' in that they lack significant computational power; Stephan showed that they cannot compute complete extensions of Peano arithmetic. 
In this announcement, we characterize the c.e.\ sets covered by, which means Turing computable from, incomplete random sets. We will see that the only c.e.\ sets that can be computed by incomplete random sets are very weak, i.e., close to being computable and very far from the halting problem. We begin by giving some background on the appropriate instance of computational weakness, $K$-triviality.

\smallskip

The $K$-trivial sets were introduced by Chaitin \cite{Chaitin_KTrivial} and first  studied by Solovay in an unpublished manuscript \cite{Solovay:manuscript}. One motivation for their definition was Chaitin's characterization \cite{Chaitin:InformationStrings} of the computable sets by the compressibility of their  initial segments. Namely, letting $C$ denote plain Kolmogorov complexity, the sets $A$ that satisfy $C(A\rest{n})\le^+ C(n)$ are the computable ones.\footnote{By $f(n) \le^+ g(n)$ we abbreviate $f(n) \le g(n) + O(1)$. We also write $f(n) \ge^+ g(n)$ with the same meaning.}  The condition means that every initial segment of $A$ contains no more information than its length, and so it is as compressible as it can be. 

The use of prefix-free Kolmogorov complexity $K$, rather than its plain variant $C$, is motivated by Schnorr's theorem (\cite{Schorr:ProcessComplexity}, see~\cite{Chaitin:ProgramSize}), which characterizes the random sets as those that have $K$-incompressible initial segments: a set $A$ is random if and only if $K(A\rest n)\ge^+ n$. The measure of an open set is determined by any prefix-free set of finite strings generating it, and so the study of prefix-free complexity, unlike the plain variety, is closely linked with measure-theoretic arguments. This explains why $K$ is the most useful Kolmogorov complexity when considering the randomness content of sets.
The $K$-trivial sets---those sets $A$ whose initial segments are as $K$-compressible as possible, in that $K(A\rest{n})\le^+ K(n)$---are the very opposite of random sets, and for that reason have sometimes been called ``anti-random''. Surprisingly, Solovay showed that in contrast with plain complexity, there are noncomputable $K$-trivial sets. 

Chaitin \cite{Chaitin_KTrivial} showed that all $K$-trivial sets are $\Delta^0_2$, that is, computable from the halting problem $\emptyset'$. In the 2000's, a series of results greatly improved our understanding of the $K$-trivial sets. These results showed that the $K$-trivial sets are all computationally very weak, that the class of $K$-trivial sets is robust, and that the notion is inherently computably enumerable. 

The computational weakness of the $K$-trivial sets was first indicated by Downey, Hirschfeldt, Nies and Stephan \cite{DHNS:Trivial}, who showed that $K$-trivial sets are incomplete;  that is, they do not compute $\emptyset'$. This result was then strengthened by Nies \cite{Nies:LownessRandomness}, who showed that in fact every $K$-trivial set is jump-traceable and superlow, roughly saying that in terms of the Turing jump operator, these sets are indistinguishable from the computable ones. Other ways to express their computational weakness in fact led \cite{Nies:LownessRandomness} to equivalent definitions of the $K$-trivial sets. The $K$-trivial sets coincide with the sets that are \emph{low for $K$}---sets that have no compression power as oracles, in that $K^A$, the prefix-free complexity relative to $A$, is equal, up to an additive constant, to the unrelativized $K$. Similarly, the $K$-trivial sets coincide with the sets that are \emph{low for ML-randomness}---sets $A$ that cannot detect any regular patterns in random sets, in that every random set is also random relative to~$A$. These equivalences witness the robustness of the class. This robustness is further reflected in the structure of the Turing degrees of $K$-trivial sets: these degrees form an ideal, and restricted to the c.e.\ sets, this ideal has a $\Sigma^0_3$ index set. Barmpalias and Nies showed that this ideal has a low$_2$ c.e.\ upper bound.  Ku\v{c}era and Slaman~\cite{Kucera.Slaman:09}  provided a low upper bound that, however, cannot be c.e.\ by a result in \cite{Muenster} (also see \cite[5.3.22]{Nies:Book}).

The relation between $K$-triviality and enumerability began with a construction, by Zambella \cite{Zambella:90}, of a noncomputable c.e.\ $K$-trivial set. Another such construction was given by Downey, Hirschfeldt, Nies and Stephan \cite{DHNS:Trivial}, which came to be known as the ``cost-function'' construction; this construction was inspired by an earlier  construction of a noncomputable c.e.\ set that is low for ML-randomness due to Ku{\v c}era and Terwijn~ \cite{Kucera.Terwijn:99}. Finally, Nies \cite{Nies:LownessRandomness} showed that $K$-triviality is an inherently enumerable notion:  every $K$-trivial set is computable from a c.e.\ $K$-trivial set (so in the Turing degrees, the ideal of $K$-trivial degrees is generated by its c.e.\ elements). He also showed that the cost-function construction is universal, in that every $K$-trivial set can be obtained as the result of an approximation that ``obeys'' the standard cost function.

These results have had several applications. Following the example of Ku\v{cera} \cite{Kucera:86}, in \cite{DHNS:Trivial} it was shown how c.e.\ $K$-trivial sets provide an injury-free solution to Post's problem. A more surprising application was given by Ku\v{c}era and Slaman \cite{KuceraSlaman:Scott}, who used $K$-trivial sets in their proof that no Scott set of Turing degrees is ``hourglass-like''. That is, for every noncomputable real in the Scott set there is a Turing incomparable real.

\smallskip

We say that a set $A$ is  a \emph{base for randomness} if $A$  is computable from some $A$-random set.  This notion was first studied in~\cite{Kucera:93}.  A noncomputable  base for randomness is so computationally weak that the cone above it, while being null, is not an effective $A$-null class.   A first indication of the relationship between $K$-triviality and incomplete random sets came from   work of Hirschfeldt, Nies and Stephan~\cite{HirschfeldtNiesStephan:UsingRandomOracles}. They showed that if $A$ is a c.e.\ set that is computable from an incomplete random set $X$, then $X$ is in fact random relative to $A$, and so $A$ is a base for randomness.   In turn, they showed that the sets that are a base for randomness coincide with  the $K$-trivial sets. Thus, every c.e.\ set that is computable from an incomplete random set  is $K$-trivial. In light of this result, in 2004, Stephan 
asked whether the c.e.\ sets that are computable from incomplete random sets are precisely the $K$-trivial sets. This problem, known as the random covering problem, became one of the major open questions in the field of algorithmic randomness. It is problem 4.6 in Miller and Nies's survey  \cite{MillerNies:Open} of open questions in the field. In light of the enumerability of $K$-trivial sets, this is equivalent to asking whether every $K$-trivial set is computable from an incomplete random set.
As already mentioned, Theorem~\ref{thm_main}  answers the question in the affirmative. This gives yet another characterization of $K$-triviality, one which uses only very basic ingredients from computability theory---namely computable enumerability, Turing reducibility, and the halting problem---together with ML-randomness (but with no reference to relativization). 

\smallskip

In retrospect, the first step toward solving the covering problem was made by Franklin and Ng who gave a  characterization of the incomplete random sets \cite{FranklinNg:Difference}. This is analogous to the Hirschfeldt--Miller result in that it gives a measure-theoretic characterization of a class of random sets defined by their interaction within the Turing degrees.  They formulated a notion of randomness---\emph{difference randomness}---and showed that it is equivalent to being random and Turing incomplete. In more detail, they showed the equivalence, for a set $Z$, of the following properties:
\begin{enumerate}
	\item $Z$ is random and incomplete.
	\item $Z$ avoids all null sets of the form $\PP \cap \bigcap_n \UU_n$, where the  $\UU_n$ are uniformly effectively open, $\PP$ is effectively closed, and $\leb(\PP\cap \UU_n) \le 2^{-n}$. 
\end{enumerate}
Franklin and Ng also showed that difference randomness could be characterized using a concept similar  to a test for Martin-L\"of randomness. In his investigations of differentiability of constructive functions on the reals, Demuth \cite{Demuth_classes_of_arithmetical} introduced notions of randomness stronger than Martin-L\"{o}f's. Like ML-randomness, his tests are null sets that are the intersection of a sequence $\seq{\UU_n}$ of effectively open sets with $\leb(\UU_n)\le 2^{-n}$. However, when defining the sets $\UU_n$ we are allowed to change our mind sometimes about what $\UU_n$ is. In Martin-L\"of tests, the function taking $n$ to an index for $\UU_n$ is computable; Demuth allowed effectively approximable functions, with a computable bound on the number of mind-changes; this notion of randomness is now known as \emph{weak Demuth} randomness (see \cite{Kucera.Nies:12} for background on Demuth's work in randomness). Franklin and Ng showed that a particular class of weak Demuth tests also captured difference randomness; in their tests, the different ``versions'' for each component $\UU_n$ have to be disjoint. 

\smallskip

The second step was made in 2011 by Bienvenu, H\"{o}lzl, Miller and Nies \cite{Bienvenu.Hoelzl.ea:12}, who gave an analytic characterization of incomplete randomness. Recall that the Lebesgue density theorem says that if $B$ is a measurable set, then for almost all $x\in B$, the limit of the conditional probability (or measure) 
\[ \leb(B|I) = \frac{\leb(B\cap I)}{\leb(I)}\]
for intervals $I$ that contain $x$ of shrinking diameter, is 1. When working in Cantor space, rather that using arbitrary intervals, it is more natural to use dyadic intervals.  The analog of Lebesgue's theorem in this context says that for any measurable set $B\subseteq 2^\w$, for almost all $X\in B$, the \emph{lower dyadic density of $B$ at $X$}, 
\[ \rho(B|X) = \liminf_{n\to \infty} \leb(B | [X\rest n]),\]
is 1; here $[\s]$ denotes the basic clopen subset of $2^\w$ determined by the finite binary string $\s$. 

Computability theorists often try to find   effective content in  results of classical mathematics. In analysis, the effective versions of almost-everywhere theorems often translate to characterizations of notions of randomness in analytic terms. Bienvenu, H\"{o}lzl, Miller and Nies~\cite{Bienvenu.Hoelzl.ea:12} applied this to a restricted form of Lebesgue's theorem and showed that the following are equivalent for a random set $X\in 2^\w$:%
\footnote{They actually showed the result for Lebesgue density in the context of the unit interval, and then derived the weaker dyadic variant (for detail see the journal version \cite[Remark 3.4]{Bienvenu.Hoelzl.ea:12a}); we restrict our attention to dyadic density.}
\begin{enumerate}
	\item $X$ is difference random; 
	\item\label{it:pd} $\rho(\PP|X)>0$ for all effectively closed sets $\PP$ containing $X$.
\end{enumerate}
If $X$ has property \eqref{it:pd}, we call it a \emph{positive density point}. Together with Franklin and Ng's work, we see that a random set $X$ is incomplete if and only if it is a positive density point. The first indication that the analytic notion of density is relevant to understanding the interaction of random and $K$-trivial sets was given by Day and Miller \cite{Day.Miller:12}. They used density and the results from \cite{Bienvenu.Hoelzl.ea:12} to solve a problem related to the covering problem, known as the ML-cupping problem; in particular, they showed that a set $A$ is $K$-trivial if and only if there is no incomplete random set $Z$ that joins $A$ above $\emptyset'$. 

\smallskip

The work in \cite{Bienvenu.Hoelzl.ea:12} left open the problem of characterizing those random sets $X$ for which the full effective version of Lebesgue's density theorem holds. We say that $X\in 2^\omega$ is a \emph{density-one point} if $\rho(\PP|X)=1$ for all effectively closed sets $\PP$ containing $X$. The question that remained was whether every random positive density point is a density-one point. Put differently, if $X$ is random and not a density-one point, must it compute $\emptyset'$? It was known by July 2011 (see Bienvenu, H\"{o}lzl, Miller and Nies~\cite{Bienvenu.Hoelzl.ea:12a}) that any such set is \emph{$LR$-hard}: every $X$-random set is actually $\emptyset'$-random. While this proves that $X$ has much of the computational strength of $\emptyset'$, it was also known (\cite[6.3.10]{Nies:Book}, also Ku\v{c}era (unpublished)) that some incomplete random sets have this highness property.

\smallskip

During and after a research-in-pairs stay at the mathematical research institute in Oberwolfach in February 2012,  Bienvenu, Greenberg, Ku\v{c}era, Nies and Turetsky found the analogue of the Hirschfeldt--Miller and Franklin--Ng characterizations for computing $K$-trivial sets. They defined a notion of randomness, called \emph{Oberwolfach randomness}. They were motivated by work by Figueira, Hirschfeldt, Miller, Ng and Nies \cite{Figueira.Hirschfeldt.ea:12}, who investigated the randomness strength of a $\Delta^0_2$ random set $Z$ by counting the number of changes required in any computable approximation of $Z$. This was linked with Demuth's idea, mentioned earlier, of accepting changes in components of tests. Oberwolfach randomness is a weak form of weak Demuth randomness, in which the changes of the components have to be coherent between the levels (the changes of $\UU_{n+1}$ are limited by the changes of $\UU_n$). In fact, examining the argument given by Franklin and Ng shows that a ``version-disjoint'' variant of Oberwolfach randomness suffices in order to capture difference randomness. Thus, Oberwolfach randomness lies between weak Demuth and difference randomness. And indeed, this notion of randomness characterized the candidates for the solution of the covering problem. 

\begin{theorem}[Bienvenu, Greenberg, Ku\v{c}era, Nies and Turetsky~\cite{Bienvenu.Greenberg.ea:preprint}]\label{thm_OW_and_K_triv} \
	\begin{enumerate}
		\item 	Suppose a set  $X$ is random. Then $X$  is not Oberwolfach random if and only if $X$  computes all $K$-trivial sets.
		\item  There is a  $K$-trivial set that is  ``smart'' in that it is not computable from any Oberwolfach random set. 
	\end{enumerate}
\end{theorem}
The ``smart'' $K$-trivial showed that if the covering problem has a positive solution, then it has a strong positive solution in that some incomplete random set would compute \emph{all} $K$-trivial sets. 

For the solution of the random covering problem, one needs    the implication from left to right in (1). The authors show that $X$ lies in a $\Pi^0_2$ null class   such that the cost function    derived from the  Hirschfeldt--Miller argument (see \cite[Th.~5.3.16]{Nies:Book}) is additive in the sense of \cite{Nies:ICM}. They then use the fact  that every $K$-trivial  $A$ obeys every additive cost function (\cite{Nies:ICM}; also see  \cite{Bienvenu.Greenberg.ea:preprint}) in order to conclude that  $A \le_T X$. 

The authors of \cite{Bienvenu.Greenberg.ea:preprint} used the technique from \cite{Bienvenu.Hoelzl.ea:12a} to show that every random that is not Oberwolfach random is $LR$-hard, and therefore high. Thus, the construction of the ``smart'' $K$-trivial set showed that no low random set can compute all $K$-trivial sets (in contrast with the existence of a low PA degree above all $K$-trivial sets). The question whether such a random set exists was a strong variant of the covering problem, which was  also posed by Stephan in 2004. Using work from \cite{Figueira.Hirschfeldt.ea:12}, they also concluded that the smart $K$-trivial set is not computable from both halves of a random set, negatively solving another strong variant of the covering problem (Problem 4.7 in~\cite{MillerNies:Open}). 

Further, the authors of \cite{Bienvenu.Greenberg.ea:preprint} made a connection between Oberwolfach randomness and Lebesgue's density theorem, by showing that if $X$ is Oberwolfach random then it is a density-one point. This was a corollary of showing that if $X$ is Oberwolfach random, then every interval-c.e.\footnotemark\ function  is differentiable at $X$ (identifying $X$ with the real that has binary expansion $0.X$).
\footnotetext{A non-decreasing lower semicontinuous function $f\colon[0,1]\to\mathbb{R}$ is \emph{interval-c.e.}\ if $f(0)=0$ and $f(y)-f(x)$ is a left-c.e.\ real, uniformly in rational numbers $x, y$. Equivalently, by work of Freer, Kjos-Hansen, Nies and Stephan~\cite{Freer.Kjos.ea:nd}, an interval c.e.\ function is the variation function of a computable real-valued function.}
As a by-product, they obtained:

\begin{theorem}[\cite{Bienvenu.Greenberg.ea:preprint}]	\label{thm_K-triv_and_density}
	If $X$ is a random set that is not a density-one point, then $X$ computes all $K$-trivial sets. 
\end{theorem}

An alternative proof of Theorem~\ref{thm_K-triv_and_density} is given in \cite{Bienvenu.Hoelzl.ea:12a}, using the same technique as the proof that such an $X$ is $LR$-hard. However, the alternative proof requires the fact that $K$-trivial sets are low for ML-randomness, hence it is not useful in the proof of Theorem~\ref{thm:TuringandlowforK} below.

\smallskip

Theorem~\ref{thm_K-triv_and_density} set the stage for the last ingredient in the solution of the covering problem, which was provided in August 2012 by Day and Miller~\cite{Day.Miller:nd}. Devising a notion of forcing using a collection of effectively closed sets especially defined to control density, they showed:

\begin{theorem}[\cite{Day.Miller:nd}] \label{thm_Day_Miller_density}
	There is a random set $X$ (in fact, a $\Delta^0_2$ one) that is a positive density point but not a density-one point.
\end{theorem}

The Franklin--Ng, Bienvenu--H\"{o}lzl--Miller--Nies, Bienvenu--Greenberg--Ku\v{c}era--Nies--Turetsky and Day--Miller results now all combine to settle the covering problem in the affirmative:

\begin{corollary}\label{cor:covering}
	There is an incomplete, $\Delta^0_2$ random set that computes all $K$-trivial sets. 
\end{corollary}

Further, Day and Miller incorporated non-$K$-trivial upper-cone avoidance with their forcing and constructed an exact pair of random sets for the ideal of $K$-trivial degrees. 

\smallskip

Apart from the inherent interest in the covering problem, and the unexpected path to its solution using analytic concepts, these results give alternative, modular proofs of some of the results concerning $K$-trivial sets. Hitherto, the fact that $K$-trivial sets are low for $K$ (or low for ML-randomness), and the fact that the $K$-trivial sets are downward closed under Turing reducibility, were proved using the \emph{decanter} argument and it stronger version, the \emph{golden run} technique. Researchers have found these highly combinatorial and complex techniques somewhat daunting. We can now provide alternative arguments. We note that (2) below  easily implies (1).
\begin{theorem}[Nies \cite{Nies:LownessRandomness}]\label{thm:TuringandlowforK} \

	\begin{enumerate}
		\item  	Let  $A$ be   $K$-trivial. Then  every set $B\le_{\Tur} A$ is $K$-trivial.
		\item  	Let  $A$ be   $K$-trivial. Then $A$ is low for $K$ (hence low for ML-randomness). 
	\end{enumerate}
\end{theorem}

\begin{proof}
	(1) A direct construction of Bienvenu's (see a  forthcoming journal paper with Downey, Merkle and Nies related to \cite{Bienvenu.Downey:09})  shows that $A$ is computable from some c.e., $K$-trivial set $C$. In \cite{Bienvenu.Greenberg.ea:preprint} it is proved  directly that every c.e., $K$-trivial set is computable from every random set that is not Oberwolfach random. Then, by the  argument above, $C$ is computable from some incomplete random set~$Z$. Now the Hirschfeldt--Nies--Stephan argument first shows that $Z$ is $C$-random, and so certainly $A$-random and $B$-random. Thus, $A$ and a fortiori  $B$,  is a  base for randomness. The ``hungry sets'' argument from \cite{HirschfeldtNiesStephan:UsingRandomOracles} now shows that $B$ is $K$-trivial. 

\noindent (2) The set  $A$ is low for $K$ by a slightly more elaborate ``hungry sets'' argument due to  \cite[Section 5.1]{Nies:Book}.
\end{proof}

\emph{Diamond classes} were  investigated   in   \cite{Nies:Book}, \cite{GreenbergNies:Benign} and elsewhere. These are collections of c.e.\ sets of the form 
\[ \CC^\diamond = \left\{ A \text{ c.e.}  \,:\,  A\le_\Tur X \text{ for all random }X\in \CC\right\},\]
where $\CC$ is any collection of sets; they are naturally ideals in the c.e.\ degrees. If $\CC$ is not null, then $\CC^\diamond$ consists of the computable sets; the Hirschfeldt--Miller argument shows that if $\CC$ is a null $\Sigma^0_3$ class then $\CC^\diamond$ contains a noncomputable set. Greenberg and Nies \cite{GreenbergNies:Benign} showed, for example, that the strongly jump-traceable c.e.\ sets form a diamond class. Results in \cite{Bienvenu.Greenberg.ea:preprint}, together with the smart $K$-trivial set, give a diamond class ($\textsf{JTH}^\diamond$) that lies strictly between the $K$-trivial and the strongly jump-traceable degrees. The covering result answers a question by Nies, by letting $\CC$ be the class of sets that are not Oberwolfach random:

\begin{corollary}\label{cor:diamond}
	The $K$-trivial c.e.\ sets form a diamond class. 
\end{corollary}

Let $X\in 2^\omega$ be Martin-L\"of random. The following diagram summarizes the properties we have discussed:

%
%

\medskip

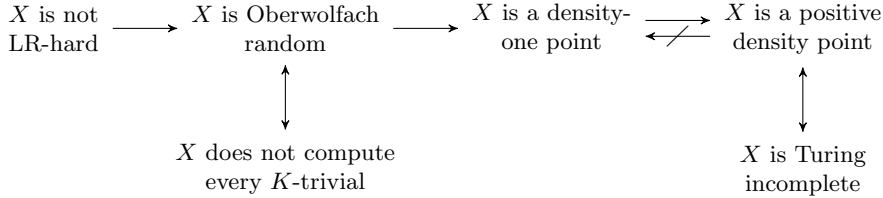
\begin{center}\small
\begin{tikzpicture}[every text node part/.style={align=center}, node distance=1cm]
	\node (lr) {$X$ is not\\LR-hard};
	\node[right=of lr] (ow) {$X$ is Oberwolfach\\random};
	\node[right=of ow] (d1) {$X$ is a density-\\one point};
	\node[right=of d1] (pd) {$X$ is a positive\\density point};
	\node[below=of ow] (kt) {$X$ does not compute\\every $K$-trivial};
	\node[below=of pd] (dr) {$X$ is Turing\\incomplete};

	\path (lr) edge[imp] (ow)
		(ow) edge[imp] (d1) edge[imp, <->] (kt)
		(pd) edge[imp, <->] (dr);
	\draw[color=white] (d1) -- (pd) coordinate[pos=0] (a0) coordinate[pos=1] (a1);
	\path ([yshift=3pt]a0) edge[imp] ([yshift=3pt]a1);
	\draw[imp, <-] ([yshift=-3pt]a0) to node {$\diagup$} ([yshift=-3pt]a1);
\end{tikzpicture}
\end{center}

\medskip

\noindent As mentioned above,  there is a Turing incomplete random that is $LR$-hard.  By \cite{Day.Miller:nd}, there is a ML-random positive density point that is not a density-one point.  The rightmost vertical equivalence is due to \cite{Bienvenu.Hoelzl.ea:12}  as discussed above. The other nontrivial implications were  obtained in \cite{Bienvenu.Greenberg.ea:preprint}.

\begin{question} {\rm  Determine which further implications hold for ML-random sets. } \end{question} 

In particular, we ask the following.

\begin{question} {\rm Is there an $LR$-hard Oberwolfach random set? Equivalently, does the collection of $K$-trivial sets strictly contain  the diamond class $\LRH^\diamond$?}  \end{question}

\begin{question} {\rm 
	What is the effective measure-theoretic characterization of the random sets for which Lebesgue's density theorem hold? Is it Oberwolfach randomness, or a weaker notion? How does it relate to $LR$-hardness?}
\end{question}

\bibliographystyle{plain}
\bibliography{BSL_bib}

\begin{thebibliography}{10}

\bibitem{Arslanov:1981}
Marat~M. Arslanov.
\newblock Some generalizations of a fixed-point theorem.
\newblock {\em Izv. Vyssh. Uchebn. Zaved. Mat.}, (5):9--16, 1981.

\bibitem{Bienvenu.Downey:09}
L.~Bienvenu and R.~Downey.
\newblock {K}olmogorov complexity and {S}olovay functions.
\newblock In {\em Symposium on Theoretical Aspects of Computer Science (STACS
  2009)}, volume 09001 of {\em Dagstuhl Seminar Proceedings}, pages 147--158,
  http://drops.dagstuhl.de/opus/volltexte/2009/1810, 2009.

\bibitem{Bienvenu.Greenberg.ea:preprint}
L.~Bienvenu, N.~Greenberg, A.\ Ku{\v{c}}era, A.~Nies, and D.~Turetsky.
\newblock ${K}$-triviality, {O}berwolfach randomness, and differentiability.
\newblock Mathematisches Forschungsinstitut Oberwolfach, Preprint Series, 2012.

\bibitem{Bienvenu.Hoelzl.ea:12a}
L.~Bienvenu, R.~Hoelzl, J.~Miller, and A.~Nies.
\newblock Demuth, {D}enjoy, and {D}ensity.
\newblock In preparation, 2012.

\bibitem{Bienvenu.Hoelzl.ea:12}
L.~Bienvenu, R.~Hoelzl, J.~Miller, and A.~Nies.
\newblock The {D}enjoy alternative for computable functions.
\newblock In {\em STACS 2012}, 2012.

\bibitem{Chaitin:ProgramSize}
Gregory~J. Chaitin.
\newblock A theory of program size formally identical to information theory.
\newblock {\em J. Assoc. Comput. Mach.}, 22:329--340, 1975.

\bibitem{Chaitin:InformationStrings}
Gregory~J. Chaitin.
\newblock Information-theoretic characterizations of recursive infinite
  strings.
\newblock {\em Theoret. Comput. Sci.}, 2(1):45--48, 1976.

\bibitem{Chaitin_KTrivial}
Gregory~J. Chaitin.
\newblock Nonrecursive infinite strings with simple initial segments.
\newblock {\em IBM Journal of Research and Development}, 21:350--359, 1977.

\bibitem{Day.Miller:12}
Adam~R. Day and Joseph~S. Miller.
\newblock Cupping with random sets.
\newblock To appear in the \emph{Proc. Amer. Math. Soc.}

\bibitem{Day.Miller:nd}
Adam~R. Day and Joseph~S. Miller.
\newblock Density, forcing, and the covering problem.
\newblock In preparation.

\bibitem{Demuth_classes_of_arithmetical}
Osvald Demuth.
\newblock Some classes of arithmetical real numbers.
\newblock {\em Comment. Math. Univ. Carolin.}, 23(3):453--465, 1982.

\bibitem{DowneyHirschfeldtBook}
Rodney~G. Downey and Denis~R. Hirschfeldt.
\newblock {\em Algorithmic randomness and complexity}.
\newblock Theory and Applications of Computability. Springer, New York, 2010.

\bibitem{DHNS:Trivial}
Rodney~G. Downey, Denis~R. Hirschfeldt, Andr\'{e} Nies, and Frank Stephan.
\newblock Trivial reals.
\newblock In {\em Proceedings of the 7th and 8th {A}sian {L}ogic
  {C}onferences}, pages 103--131, Singapore, 2003. Singapore Univ. Press.

\bibitem{Figueira.Hirschfeldt.ea:12}
S.~Figueira, D.~Hirschfeldt, J.~Miller, Selwyn Ng, and A~Nies.
\newblock Counting the changes of random {$\Delta^0_2$} sets.
\newblock pages 1--10, 2012.
\newblock Journal version to appear in the \emph{Journal of Logic and
  Computation}.

\bibitem{FranklinNg:Difference}
Johanna N.~Y. Franklin and Keng~Meng Ng.
\newblock Difference randomness.
\newblock {\em Proc. Amer. Math. Soc.}, 139(1):345--360, 2011.

\bibitem{Freer.Kjos.ea:nd}
C.~Freer, B.~Kjos-Hanssen, A.~Nies, and F.~Stephan.
\newblock Effective aspects of {L}ipschitz functions.
\newblock {S}ubmitted.

\bibitem{Gacs:86}
Peter G\'{a}cs.
\newblock Every sequence is reducible to a random one.
\newblock {\em Inform. and Control}, 70:186--192, 1986.

\bibitem{GreenbergNies:Benign}
Noam Greenberg and Andr{\'e} Nies.
\newblock Benign cost functions and lowness properties.
\newblock {\em J. Symbolic Logic}, 76(1):289--312, 2011.

\bibitem{HirschfeldtNiesStephan:UsingRandomOracles}
Denis~R. Hirschfeldt, Andr{\'e} Nies, and Frank Stephan.
\newblock Using random sets as oracles.
\newblock {\em J. Lond. Math. Soc. (2)}, 75(3):610--622, 2007.

\bibitem{Kucera.Slaman:09}
A.~Ku{\v{c}}era and T.~Slaman.
\newblock Low upper bounds of ideals.
\newblock {74}:{517--534}, 2009.

\bibitem{Kucera.Terwijn:99}
A.~Ku{\v{c}}era and S.~Terwijn.
\newblock Lowness for the class of random sets.
\newblock 64:1396--1402, 1999.

\bibitem{Kucera:85}
Anton{\'{\i}}n Ku{\v{c}}era.
\newblock Measure, {$\Pi^0_1$}-classes and complete extensions of {${\rm PA}$}.
\newblock In {\em Recursion theory week ({O}berwolfach, 1984)}, volume 1141 of
  {\em Lecture Notes in Math.}, pages 245--259. Springer, Berlin, 1985.

\bibitem{Kucera:86}
Antonin Ku{\v{c}}era.
\newblock An alternative, priority-free, solution to {P}ost's problem.
\newblock In {\em Mathematical foundations of computer science, 1986
  (Bratislava, 1986)}, volume 233 of {\em Lecture Notes in Comput. Sci.}, pages
  493--500. Springer, Berlin, 1986.

\bibitem{Kucera:93}
Anton{\'{\i}}n Ku{\v{c}}era.
\newblock On relative randomness.
\newblock {\em Ann. Pure Appl. Logic}, 63(1):61--67, 1993.
\newblock 9th International Congress of Logic, Methodology and Philosophy of
  Science (Uppsala, 1991).

\bibitem{KuceraSlaman:Scott}
Anton{\'{\i}}n Ku{\v{c}}era and Theodore~A. Slaman.
\newblock Turing incomparability in {S}cott sets.
\newblock {\em Proc. Amer. Math. Soc.}, 135(11):3723--3731, 2007.

\bibitem{Kucera.Nies:12}
Anton\'{\i}n Ku\v{c}era and Andr{\'e} Nies.
\newblock Demuth's path to randomness.
\newblock In {\em Proceedings of the 2012 international conference on
  Theoretical Computer Science: computation, physics and beyond}, WTCS'12,
  pages 159--173, Berlin, Heidelberg, 2012. Springer-Verlag.

\bibitem{MillerNies:Open}
Joseph~S. Miller and Andr{\'e} Nies.
\newblock Randomness and computability: open questions.
\newblock {\em Bull. Symbolic Logic}, 12(3):390--410, 2006.

\bibitem{Muenster}
A.\ Nies.
\newblock Reals which compute little.
\newblock In {\em Logic Colloquium '02}, Lecture Notes in Logic, pages
  260--274, 2002.

\bibitem{Nies:ICM}
A.\ Nies.
\newblock Interactions of computability and randomness.
\newblock In {\em Proceedings of the International Congress of Mathematicians},
  pages 30--57. World Scientific, 2010.

\bibitem{Nies:LownessRandomness}
Andr\'{e} Nies.
\newblock Lowness properties and randomness.
\newblock {\em Adv. Math.}, 197(1):274--305, 2005.

\bibitem{Nies:Book}
Andr{\'e} Nies.
\newblock {\em Computability and randomness}, volume~51 of {\em Oxford Logic
  Guides}.
\newblock Oxford University Press, Oxford, 2009.

\bibitem{Schorr:ProcessComplexity}
Claus-Peter Schnorr.
\newblock Process complexity and effective random tests.
\newblock {\em J. Comput. System Sci.}, 7:376--388, 1973.
\newblock Fourth Annual ACM Symposium on the Theory of Computing (Denver,
  Colo., 1972).

\bibitem{Solovay:manuscript}
Robert~M. Solovay.
\newblock Draft of paper (or series of papers) related to {C}haitin's work.
\newblock IBM Thomas J. Watson Research Center, Yorktown Heights, NY, 215
  pages, 1975.

\bibitem{Stephan:06}
Frank Stephan.
\newblock Marin-{L}\"of random and {PA}-complete sets.
\newblock In {\em Logic Colloquium '02}, volume~27 of {\em Lect. Notes Log.},
  pages 342--348. Assoc. Symbol. Logic, La Jolla, CA, 2006.

\bibitem{Zambella:90}
D.~Zambella.
\newblock On sequences with simple initial segments.
\newblock ILLC technical report ML 1990-05, Univ. Amsterdam, 1990.

\end{thebibliography}

\end{document}